\documentclass[a4paper,12pt,reqno]{amsart}
\usepackage[latin1]{inputenc}
\usepackage[T1]{fontenc}
\usepackage[english]{babel}
\usepackage{a4wide}
\usepackage{amsmath}
\usepackage{graphicx}
\usepackage{amssymb}

\title{Homogeneous Fractional Embeddings}

\author{Pierre Inizan}
\address{IMCCE - Observatoire de Paris \\ 77, Avenue Denfert-Rochereau \\ 75014 Paris, France}
\email{pierre.inizan@imcce.fr}

\begin{document}

\maketitle


\newcommand{\transposee}[1]{{\vphantom{#1}}^{\mathit t}{#1}}
\newcommand{\Drl}[3]{{}_{#1} \mathcal{D}_{#2}^{#3}}
\newcommand{\Dc}[3]{{}^C_{#1} \mathcal{D}_{#2}^{#3}}
\newcommand{\Dabm}{\mathcal{D}^{\alpha, \beta}_{\mu}}
\newcommand{\Diabm}{\mathbb{D}^{\alpha, \beta}_{\mu}}
\newcommand{\Dbamm}{\mathcal{D}^{\beta, \alpha}_{-\mu}}
\newcommand{\Dibamm}{\mathbb{D}^{\beta, \alpha}_{-\mu}}
\newcommand{\Da}{\mathbb{D}^{\alpha}}
\newcommand{\Db}{\mathbb{D}^{\beta}}
\newcommand{\Dae}{\mathbb{D}^{\alpha}_*}
\newcommand{\Dbe}{\mathbb{D}^{\beta}_*}
\newcommand{\ind}[5]{ {}^{#2}_{#4} {#1}_{#5}^{#3} }
\newcommand{\Drll}[0]{{}_{a} \mathcal{D}_{t}^{\alpha}}
\newcommand{\Drlr}[0]{{}_{t} \mathcal{D}_{b}^{\beta}}
\newcommand{\Dcl}[0]{{}^C_{a} \mathcal{D}_{t}^{\alpha}}
\newcommand{\Dcr}[0]{{}^C_{t} \mathcal{D}_{b}^{\beta}}
\newcommand{\Dcabm}{{}^C \mathcal{D}^{\alpha, \beta}_{\mu}}
\newcommand{\Dcbamm}{{}^C \mathcal{D}^{\beta, \alpha}_{-\mu}}
\newcommand{\Dpar}[2]{\dfrac{\partial #1}{\partial #2}}
\newcommand{\Eab}[0]{{}_a^\alpha E_b^\beta}
\newcommand{\Etab}[0]{{}_a^\alpha \tilde{E}_b^\beta}
\newcommand{\Reste}[7]{{}^{#1,#2}_{#3} \left[ #6, #7 \right]_{#4}^{#5}}
\newcommand{\Lc}[0]{\mathcal{L}}
\newcommand{\Varnab}[0]{Var^{\alpha, \beta}(a,b)}
\newcommand{\VarCab}[0]{{}^C Var(a,b)}
\newcommand{\VarCnab}[0]{{}^C Var^{\alpha, \beta}(a,b)}
\newcommand{\psc}[2]{\langle #1 , #2 \rangle}
\newcommand{\Ker}[0]{\mathcal{K}er \,}
\newcommand{\Ima}[0]{\mathcal{I}m \,}
\newcommand{\Aabm}[0]{\mathcal{A}^{\alpha, \beta}_{\mu}}
\newcommand{\Acabm}[0]{{}^C \mathcal{A}^{\alpha, \beta}_{\mu}}

\newcommand{\R}[0]{\mathbb{R}}
\newcommand{\mc}[1]{\mathcal{#1}}


\newtheorem{theorem}{Theorem}
\newtheorem{definition}{Definition}
\newtheorem{lemma}{Lemma}


\begin{abstract}
\noindent \emph{
Fractional equations appear in the description of the dynamics of various physical systems. For Lagrangian systems, the embedding theory developped by Cresson [``Fractional embedding of differential operators and Lagrangian systems'', J. Math. Phys. 48, 033504 (2007)] provides a univocal way to obtain such equations, stemming from a least action principle.
However, no matter how equations are obtained, the dimension of the fractional derivative differs from the classical one and may induce problems of temporal homogeneity in fractional objects. 
In this paper, we show that it is necessary to introduce an extrinsic constant of time. Then, we use it to construct two equivalent fractional embeddings which retain homogeneity. The notion of fractional constant is also discussed through this formalism. Finally, an illustration is given with natural Lagrangian systems, and the case of the harmonic oscillator is entirely treated.
}

\end{abstract}


\section{Introduction}

For about twenty years, the fractional calculus has known significant development, according to its successes in various domains, such as chaotic dynamics, viscoelasticity, acoustics, electricity or polymer chemistry \cite{Bag_Tor_83, Bag_tor_86, Carp_Main,Hilfer_App, Samko}.
In this context, fractional equations are used for modeling physical phenomena. No more considered as formal objects, but as physically significant, one may expect them to be homogeneous, a fundamental principle of the physical equations. In other words, the objects of a same equation must have the same dimension, expressible in terms of fundamental units, such as length $L$, mass $M$ or time $T$. For example, the dimension of a force $F$ is $[F] = M \, L^2 \, T^{-2}$.

However, in many fractional equations we can find in litterature, homogeneity is not respected. Let firm up details of our topic: from a strictly mathematical point of view, this notion of homogeneity is meaningless. The problem arises only if the equations considered are susceptible to model a physical system. But in that case, this question is fundamental. However, the problem is often hidden in modelisations because the constants are seen as parameters to fit.

How to obtain fractional equations in time ? The usual approach consists in subsistuting in the classical equation considered the derivative $d/dt$ (homogeneous to $T^{-1}$) by a new one, $d^\alpha / dt^\alpha$ (homogeneous to $T^{-\alpha}$), with $0 < \alpha < 1$. Clearly, this transformation does not conserve the temporal homogeneity of the equation. This remark would apply to fractional derivatives in space, but in this article, only the temporal ones will be considered.

The problem of the passage from the classical case to the fractional one has already been discussed in \cite{Cresson}. Shifting the derivative leads to fractional equations, but, from a physical point of view, is it the right procedure ? Do the fractional equations obtained remain physically relevant ? Answers are given in the case of Lagrangian systems. Extending the least action principle, Cresson \cite{Cresson} provides an univocal method, the fractional embedding, which leads to fractional equations. The respect of this other fundamental principle makes the equations susceptible to conserve a physical meaning.

This article is in keeping with this philosphy. In section \ref{extrinsic}, it is shown that an extrinsic characteristic time has to be introduced. After having presented the notion of embedding in section \ref{embeddings}, we show in section \ref{homogeneous} that this time constant leads to two equivalent methods of fractional embedding, which additionally conserve homogeneity. The notion of fractional constant is also adressed. Finally, an illustration is given in section \ref{natural} with the case of natural Lagrangian systems. A complete resolution is given for the harmonic oscillator.


\section{Extrinsic characteristic time} \label{extrinsic}

In this paper, the notation $d^\alpha / dt^\alpha$ will concern an unspecified fractional derivative, except in section \ref{sec_osc}. Moreover, the difficulties inherent to fractional derivatives (for example, $d^\alpha / dt^\alpha \circ d^\alpha / dt^\alpha \neq d^{2\alpha} / dt^{2\alpha}$ in general) won't be addressed here.

As it has already been said, the fractional aspect of equations occurs with derivatives generalized to noninteger orders. If this sustitution is done in one term of an equation, the other ones have also to be modified so as to preserve homogeneity. 
As a first example, we consider the damped oscillator equation,
\begin{equation} \label{oscillateur}
\dfrac{d^2}{dt^2} x(t) + \lambda \dfrac{d}{dt} x(t) + \omega^2 x(t)= 0.
\end{equation} 

The constants $\lambda$ and $\omega$ are homogeneous to $T^{-1}$.
Replacing $d / dt$ by $d^\alpha / dt^\alpha$ gives
\begin{equation}
\dfrac{d^{2\alpha}}{dt^{2\alpha}} x(t) + \lambda \dfrac{d^\alpha}{dt^\alpha} x(t) + \omega^2 x(t) = 0. \nonumber
\end{equation} 

This equation is non-homogeneous: the three terms have different temporal dimensions ($T^{-\alpha}$, $T^{-(1 + \alpha)}$ and $T^{-2}$).
A homogeneous version could be
\begin{equation} \label{oscill_1}
\dfrac{d^{2\alpha}}{dt^{2\alpha}} x(t) + \lambda^\alpha \dfrac{d^\alpha}{dt^\alpha} x(t) + \omega^{2 \alpha} x(t) = 0.
\end{equation} 

Things are getting trickier with the diffusion equation,
\begin{equation} \label{diffusion_clas}
\dfrac{\partial}{\partial t} u(x,t) - D \dfrac{\partial^2}{\partial x^2} u(x,t) = 0, \text{ with } [D] = L^2 \, T^{-1}.
\end{equation} 

If the operator $\partial / \partial t$ is directly remplaced by $\partial^\alpha / \partial t^\alpha$, the only way to preserve the temporal homogeneity is to shift $D$ into $D^\alpha$. We obtain
\begin{equation}
\dfrac{\partial^\alpha}{\partial t^\alpha} u(x,t) - D^\alpha \dfrac{\partial^2}{\partial x^2} u(x,t) = 0. \nonumber
\end{equation} 

Unfortunately, this equation becomes now nonhomogeneous in space: $\partial^\alpha / \partial t^\alpha$ has no spatial dimension, whereas the spatial dimension of $D^\alpha \, \partial^2 / \partial x^2$ is $L^{2(1-\alpha)}$.
The only way to preserve the spatial homogeneity is to change $\partial / \partial x$ by $\partial^\alpha / \partial x^\alpha$, which leads to
\begin{equation} \label{diff_alpha}
\dfrac{\partial^\alpha}{\partial t^\alpha} u(x,t) - D^\alpha \dfrac{\partial^{2\alpha}}{\partial x^{2\alpha}} u(x,t) = 0.
\end{equation} 

Regarding the numerous studies dealing with that equation (for example, \cite{Meer_Schef, Zasl_HCFD}), one of the most used in fractional calculus, this case seems too restrictive. According to \cite[chap. 16]{Zasl_HCFD}, anomalous transport cannot be described by \eqref{diff_alpha} (because the transport exponent remains classical).  

As a last example, we consider the free fall equation,
\begin{equation} \label{chute_libre}
\dfrac{d^2}{dt^2} x(t) + g = 0.
\end{equation} 

The dimension of the gravity acceleration $g$ is $[g] = L \, T^{-2}$.

If $d^2 /dt^2 $ is changed directly into $d^{2\alpha} / dt^{2\alpha}$, it is impossible to obtain an fractional and homogeneous equation.

A first conclusion could be drawn: a necessary condition for an equation to be physically relevant in the fractional case is the possibility of obtaining a homogeneous formulation. In this case, the damped oscillator and the diffusion fractional equations fill this condition, unlike the free fall equation. 

However, the constants which appear the equations may be built from other constants. For example, the gravity acceleration $g$ is in fact defined by $g = \mathcal{G} \, M_T \, R_T^{-2}$, where $\mathcal{G}$ is the gravity constant, $M_T$ et $R_T$ the mass and the radius of the Earth. So, equation \eqref{chute_libre} can provide a fractional and homogeneous equation,
\begin{equation}
\dfrac{d^{2\alpha}}{dt^{2\alpha}} x(t) + g^\alpha R_T^{1-\alpha} = 0. \nonumber
\end{equation} 

Therefore, the initial equation itself is not suffisant to conclude if the fractional equation derived is physically relevant or not. This first conclusion is consequently invalid.

Moreover, if no restrictions are added, there may be several ways to obtain fractional equations, which is unsatisfactory. Indeed, \eqref{oscillateur} can notably turn into
\begin{equation}
\dfrac{d^{2\alpha}}{dt^{2\alpha}} x(t) + \lambda^\alpha \dfrac{d^\alpha}{dt^\alpha} x(t) + \omega^2 \lambda^{2\alpha - 2} x(t) = 0, \nonumber
\end{equation} 
or into
\begin{equation} 
\dfrac{d^{2\alpha}}{dt^{2\alpha}} x(t) + \lambda^{2\alpha-\beta} \dfrac{d^\beta}{dt^\beta} x(t) + \omega^{2 \alpha} x(t) = 0, \quad 0 < \beta < 1. \nonumber
\end{equation} 
Both of these equations differ from \eqref{oscill_1}.

Facing this problem of underdetermination, equations have to be constrained so as to provide a unique fractional and homogeneous equation. As it was mentioned above, the fractional embedding plays this part in the case of Lagrangian systems. It implies, among others, that there is a single fractional exponent in fractional equations describing such systems. Furthermore, it does not modify the constants of the initial equation. Once the Lagrangian is given, the structure is totally determined and the fractional aspect appears only through the switch of $d / dt$ into $d^\alpha / dt^\alpha$. Consequently, a suitable fractional operator should not modify the homogeneity of the initial equation. Two possibilities are conceivable. The first one is the temporal nondimensionalization of the system: the classical and fractional temporal operators become dimensionless, and the homogeneity of the equation is preserved. The second method consists in choosing a fractional operator which dimension is the same as the classical one, i.e. $T^{-1}$. 
Both of them require the introduction of a time constant $\tau$. The examples above show that $\tau$ may not be determined directly by the equation (see \eqref{chute_libre}), or that several constants could suit for $\tau$ (see \eqref{oscillateur}). Consequently, the time constant $\tau$ has to be postulated, and defined \emph{a priori}, i.e., independently of the classical equation. The introduction of such a constant enables to solve the problem of homogeneity, while remaining compatible with the fractional embedding.


\section{General embeddings} \label{embeddings}

We adapt here the embedding notion presented in \cite{Cresson}.


\subsection{Embedding of differential operators}

Let $\textbf{f} = (f_1, \ldots, f_p)$ and $\textbf{g} = (g_1, \ldots, g_p)$ be two $p$-uplets of $C^\infty$ functions $\R^{k+2} \longrightarrow \R$. Let $a,b \in \R$ with $a < b$.
We denote by $\mc{O} (\textbf{f},\textbf{g})$ the differential operator defined by
\begin{equation} \label{Ofg}
\begin{array}{cccl}
\mc{O} (\textbf{f},\textbf{g}) \, : & C^\infty([a,b]) & \longrightarrow & \quad C^\infty([a,b]) \\
             &       x		& \longmapsto    &  \displaystyle{\sum_{i=0}^p} \left( f_i \cdot \left( \dfrac{d}{dt} \right)^i \circ g_i \right) \left( x(\bullet), \ldots,  \left( \dfrac{d}{dt} \right)^k x(\bullet), \bullet \right),
\end{array}
\end{equation} 

where, for any function $u \, : \, \R^{k+2} \longrightarrow \R$, any functions $x_0, \ldots, x_k \; : \, [a,b] \longrightarrow \R$,
\begin{equation}
\begin{array}{cccl}
u \left( x_0(\bullet), \ldots, x_k(\bullet), \bullet \right) \, : & [a,b] & \longrightarrow & \; \R \\
            						 &       t	& \longmapsto    & u(x_0(t),\ldots, x_k(t),t), \nonumber
\end{array}
\end{equation} 

and, for any functions $f$ and $g$, $(f \cdot g) (t) = f(t) \cdot g(t)$.

Now we extend this class of operators with an operator which generalizes $d / dt$.

\begin{definition}
Let $\mc{D} \, : \, C^\infty([a,b]) \longrightarrow C^\infty([a,b])$ be a differential operator. With the previous notations, the $\mc{D}$-embedding of $\mc{O} (\textbf{f},\textbf{g})$, denoted $\mc{E}(\mc{O} (\textbf{f},\textbf{g}),\mc{D})$, is defined by
\begin{equation}
\begin{array}{cccl}
\mc{E}(\mc{O} (\textbf{f},\textbf{g}),\mc{D}) \, : & C^\infty([a,b]) & \longrightarrow & \quad C^\infty([a,b]) \\
             &       x		& \longmapsto    &  \displaystyle{\sum_{i=0}^p} \left( f_i \cdot \mc{D}^i \circ g_i \right) \left(  x(\bullet), \ldots,  \mc{D}^k x(\bullet), \bullet \right). \nonumber
\end{array}
\end{equation}   

\end{definition}

We define the ordinary differential equation associated with $\mc{O} (\textbf{f},\textbf{g})$ by
\begin{equation} \label{ODE}
\mc{O} (\textbf{f},\textbf{g})(x) = 0, \quad x \in C^\infty([a,b]).
\end{equation} 

\begin{definition}
Let $\mc{D} \, : \, C^\infty([a,b]) \longrightarrow C^\infty([a,b])$ be a differential operator. With the previous notations, the $\mc{D}$-embedding of equation \eqref{ODE} is defined by
\begin{equation}
\mc{E}(\mc{O} (\textbf{f},\textbf{g}),\mc{D}) (x) = 0, \quad x \in C^\infty([a,b]). \nonumber
\end{equation} 
\end{definition}


\subsection{Embedding of Lagrangian systems}

We consider a Lagrangian system, with Lagrangian $L(x,v,t)$, $t \in [a,b]$, and $\mc{D} \, : \, C^\infty([a,b]) \longrightarrow C^\infty([a,b])$ a differential operator.

The Lagrangian $L$ can naturally lead to a differential operator of the form \eqref{Ofg}
\begin{equation}
\begin{array}{cccl}
\mc{O} (1,L) \, : & C^\infty([a,b]) & \longrightarrow & \quad C^\infty([a,b]) \\
             &       x		& \longmapsto    &  L \left( x(\bullet), \dfrac{d}{dt} x(\bullet), \bullet \right). \nonumber
\end{array}
\end{equation}

Now we identify $L$ and $\mc{O}(1,L)$.

The $\mc{D}$-embedding of $L$, $\mc{E}(L, \mc{D})$, will be denoted $\hat{L}(\mc{D})$ or shortened into $\hat{L}$,
\begin{equation}
\begin{array}{cccl}
\hat{L}(\mc{D}) \, : & C^\infty([a,b]) & \longrightarrow & \quad C^\infty([a,b]) \\
             &       x		& \longmapsto    &  L( x(\bullet), \mc{D} \, x(\bullet), \bullet). \nonumber
\end{array}
\end{equation}

In Lagrangian mechanics, the action and its minima play a central role. Let us define an action suitable for the embedding formalism.

\begin{definition}
Let $g \, : \, C^\infty([a,b]) \longrightarrow C^\infty([a,b])$ be a mapping. The action of $g$, denoted $\mc{A}(g)$, is defined by
\begin{equation}
\begin{array}{cccl}
\mc{A}(g) \, : & C^\infty([a,b]) & \longrightarrow & \quad \R \\
             &       x		& \longmapsto    &  \displaystyle{\int_a^b} g(x)(t) \, dt. \nonumber
\end{array}
\end{equation} 
\end{definition}

For example, with the identification $L \equiv \mc{O}(1,L)$, the action of $L$ is given by
\begin{equation}
\mc{A}(L)(x) = \int_a^b L \left( x(t), \dfrac{d}{dt} x(t), t \right) \, dt. \nonumber
\end{equation} 

Concerning the $\mc{D}$-embedding of $L$, the associated action is
\begin{equation}
\mc{A}(\hat{L}(\mc{D}))(x) = \int_a^b L \left( x(t), \mc{D} x(t), t \right) \, dt. \nonumber
\end{equation} 

The extremum of the associated action to a Lagrangian system provides the equation of motion of the system trough the following theorem.

\begin{theorem}
The action $\mc{A}(L)$ is extremal in $x$ if and only if $x$ satisfies the Euler-Lagrange equation, given by
\begin{equation} \label{EL_clas}
\forall t \in [a,b], \quad \partial_1 L \left( x(t), \dfrac{d}{dt} x(t), t \right) - \dfrac{d}{dt}  \partial_2 L \left( x(t), \dfrac{d}{dt} x(t), t \right) = 0.
\end{equation}
\end{theorem}

Equation \eqref{EL_clas} is denoted $EL(L)$.
For the embedded Lagrangian, a connected result states the following.

\begin{theorem}
Under some appropriate variations of the action (see \cite{Inizan_CBHF}), $\mc{A}(\hat{L}(\mc{D}))$ is extremal in $x$ if $x$ satisfies the causal Euler-Lagrange equation, given by
\begin{equation} \label{EL_gen}
\forall t \in [a,b], \quad \partial_1 L(x(t), \mc{D} x(t), t) - \mc{D} \, \partial_2 L( x(t), \mc{D} x(t), t) = 0.
\end{equation}
\end{theorem}

Equation \eqref{EL_gen} is denoted $EL(\hat{L}(\mc{D}))$. 
This leads to the following result.

\begin{theorem}
In the causal case (i.e. with the appropriate variations of the action), we have the commutative scheme,
\begin{equation}
EL(\mc{E}(L, \mc{D})) \equiv \mc{E}(EL(L),\mc{D}). \nonumber
\end{equation} 
This embedding procedure is qualified as \emph{coherent}.
\end{theorem}

\begin{proof}

Equation \eqref{EL_clas} can be written in the form \eqref{ODE}, with $\textbf{f}= (1,1)$ and $\textbf{g}=(\partial_1 L, - \partial_2 L)$. Its $\mc{D}$-embedding is exactly \eqref{EL_gen}.

\end{proof}


\section{Homogeneous fractional embeddings} \label{homogeneous}

Let us now explain the two methods presented in the Introduction for the Lagrangian sytems. We show that they are equivalent. Then we present what could be a third method, also equivalent, which is farther from the philosophy of the embedding theory, but which makes fractional constants to appear.

Let $L(x,v,t)$ be a Lagrangian, with $t \in [a,b]$, and $\tau$ the time constant introduced in the previous part.
The temporal dimensions of the variables $x$, $v$ and $t$ are respectively $T^0$, $T^{-1}$ and $T$.


\subsection{Temporal nondimensionalization}

Nondimensionalization is a widespread method used in physics, particularly in fluid mechanics, to simplify equations and to exhib relevant parameters. Concerning fractional dynamics, the constant $\tau$ introduced above seems physically relevant. In the following part, it is used to nondimensionalize equations according to the temporal dimension. 

\begin{definition}
We note $\tilde{a}=a / \tau$ and $\tilde{b}= b / \tau$.
The temporal nondimensionalized Lagrangian, denoted $L_n$, evolving on the interval $[\tilde{a}, \tilde{b}]$, is defined by
\begin{equation}
\forall u \in [\tilde{a}, \tilde{b}], \quad L_n (x,y,u) = L( x, \dfrac{y}{\tau}, \tau u ). \nonumber
\end{equation} 
\end{definition}

None of the variables of $L_n$ has temporal dimension, which justify the denomination.
We choose here $\mc{D} = d^\alpha / dt^\alpha$.
As the new temporal evolution variable $u$ has no dimension, the classical derivative according to this new time, $d / du$, is also dimensionless, such as $d^\alpha / du^\alpha$.
Therefore, the substitution $d / du \rightarrow d^\alpha / du^\alpha$ does not break homogeneity anymore.

For example, if $d^\alpha / dt^\alpha$ is taken as the Caputo derivative, defined by
\begin{equation}
\Drl{s_0}{s}{\alpha} f(s) = \dfrac{1}{\Gamma(1-\alpha)} \int_{s_0}^s (s-v)^{-\alpha} f'(v) \, dv, \nonumber
\end{equation} 

we see that if $t_0$ and $t$ have the dimension of a time, $d^\alpha / dt^\alpha = \Drl{t_0}{t}{\alpha}$ is homogeneous to $T^{-\alpha}$, whereas if $u_0$ and $u$ are dimensionless, $d^\alpha / du^\alpha = \Drl{u_0}{u}{\alpha}$ is dimensionless.

The $d^\alpha / du^\alpha$-embedding of $L_n$, $\hat{L}_n \left( d^\alpha / du^\alpha \right)$, is
\begin{equation}
\begin{array}{cccl}
\hat{L}_n \left( \dfrac{d^\alpha}{du^\alpha} \right) \, : & C^\infty([\tilde{a},\tilde{b}]) & \longrightarrow & \quad C^\infty([\tilde{a},\tilde{b}]) \\
             &       x		& \longmapsto    &  L_n \left( x(\bullet), \dfrac{d^\alpha}{du^\alpha} x(\bullet), \bullet \right). \nonumber
\end{array}
\end{equation}   

For sake of lisibility, $\hat{L}_n \left( d^\alpha / du^\alpha \right)$ will be shortened into $\hat{L}_n$.
It verifies
\begin{equation}
\forall x \in C^\infty([\tilde{a},\tilde{b}]), \, \forall u \in [\tilde{a},\tilde{b}], \quad
\hat{L}_n (x)(u) = L \left( x(u), \tau^{-1} \dfrac{d^\alpha}{du^\alpha} x(u), \tau u \right). \nonumber
\end{equation}

The associated action to $\hat{L}_n$ is
\begin{equation}
\mc{A}(\hat{L}_n)(x) = \int_{\tilde{a}}^{\tilde{b}} L_n \left( x(u), \dfrac{d^\alpha}{du^\alpha} x(u),u \right) \, du \nonumber
\end{equation} 

and leads to the following result.

\begin{theorem}
The Euler-Lagrange equation associated with the $d^\alpha / du^\alpha$-embedding of the nondimensionalized Lagrangian $L_n$ is given, in the causal case, by
\begin{equation} \label{EL1}
\forall u \in [\tilde{a}, \tilde{b}], \quad  \partial_1 L_n \left( x(u), \dfrac{d^\alpha}{du^\alpha} x(u), u \right) - \dfrac{d^\alpha}{du^\alpha} \partial_2 L_n \left( x(u), \dfrac{d^\alpha}{du^\alpha} x(u), u \right)= 0, 
\end{equation}
with $x \in C^\infty([\tilde{a},\tilde{b}])$.
\end{theorem}

This is the equation of motion of the Lagrangian system which Lagrangian is $L_n$, and which dynamics evolution is fractional, according to the operator $d^\alpha / du^\alpha$. The solution $x$ of \eqref{EL1} depends on the nondimensionalized evolution variable $u$.
As $d^\alpha / du^\alpha$ has no dimension, equation \eqref{EL1} is clearly homogeneous.

This method eludes the question of the dimension of the fractional operator, and, by modifying the Lagrangian variables (but not the Lagrangian itself), it leads to an homogeneous fractional embedding. However, it occults the real dynamics, based on a real (i.e., dimensionalized) time. The following method will lift this veil.


\subsection{Homogeneous fractional derivative}

The time constant $\tau$ can be used to nondimensionalize the temporal operator, but it also enables to build a fractional operator which dimension remains $T^{-1}$: $\tau^{\alpha-1} d^\alpha / dt^\alpha$. Therefore, this operator preserves the homogeneity of the fractional embedding. Let us detail the procedure.

We consider the same initial Lagrangian $L(x,v,t)$, with $t \in [a,b]$. Contrary to the previous method, there is no need to define a new Lagrangian nor a new evolution variable.
Now the differential operator $\mc{D}$ becomes $\tau^{\alpha-1} d^\alpha / dt^\alpha$ and acts on functions which depend on the real time $t$.

The $\tau^{\alpha-1} d^\alpha / dt^\alpha$-embedding of $L$, $\hat{L} \left( \tau^{\alpha-1} d^\alpha / dt^\alpha \right)$ will be denoted $\hat{L}_h$.
\begin{equation}
\begin{array}{cccl}
\hat{L}_h \, : & C^\infty([a,b]) & \longrightarrow & C^\infty([a,b]) \\
             &       x		& \longmapsto    &  L \left( x(\bullet), \tau^{\alpha-1} \dfrac{d^\alpha}{dt^\alpha} x(\bullet), \bullet \right), \nonumber
\end{array}
\end{equation}

The associated action $\mc{A}(\hat{L}_h)$ verifies
\begin{equation}
\forall x \in C^\infty([a,b]), \quad \mc{A}(\hat{L}_h)(x) = \int_a^b L \left( x(t), \tau^{\alpha-1} \dfrac{d^\alpha}{dt^\alpha} x(t), t \right) \, dt. \nonumber
\end{equation}

The equation of motion is now given by
\begin{theorem}
The Euler-Lagrange equation associated with the $\left( \tau^{\alpha-1} d^\alpha / dt^\alpha \right)$-embedding of $L$ is given, in the causal case, by
\begin{equation} \label{EL2}
\forall t \in [a, b], \quad \partial_1 L \left( x(t), \tau^{\alpha-1} \dfrac{d^\alpha}{dt^\alpha} x(t),t \right) - \tau^{\alpha-1} \dfrac{d^\alpha}{dt^\alpha} \partial_2 L \left( x(t), \tau^{\alpha-1} \dfrac{d^\alpha}{dt^\alpha} x(t),t \right)  = 0.
\end{equation} 
\end{theorem}

In this method, the fractional Euler-Lagrange equation and its solutions $x$ depend of the real time $t$, and all of the fractional objects (variables, operator, Lagrangian, equation) hold the same dimension as their classical predecessors.


\subsection{Equivalence between the two methods}

Both of these two methods constitute homogeneous fractional embeddings. Fortunately, if a quite natural condition is respected, they turn out to be equivalent.

We suppose that the fractional derivative $d^\alpha / dt^\alpha$ verifies the following condition:
\begin{equation} \label{D}
\forall x \in C^\infty([a,b]), \quad \tau^\alpha \dfrac{d^\alpha}{dv^\alpha} \left. x(v) \right|_{v=t} = \dfrac{d^\alpha}{du^\alpha} \left. \tilde{x}(u) \right|_{u = t/\tau}, \quad \text{where } \tilde{x} \, : \, u \mapsto x(u\tau).
\end{equation} 

This property is notably verified by the usual fractional derivatives. For example, for the Caputo derivative, we have
\begin{eqnarray*}
\Drl{\tilde{a}}{u}{\alpha} \tilde{x}(u) & = & \dfrac{1}{\Gamma(1-\alpha)} \int_{\tilde{a}}^u (u-v)^{-\alpha} \tilde{x}'(v) \, dv \\
			& = & \dfrac{1}{\Gamma(1-\alpha)} \int_{\tilde{a} \tau}^{u \tau} (u-\dfrac{s}{\tau})^{-\alpha} \tilde{x}'(\dfrac{s}{\tau}) \, \dfrac{ds}{\tau}, \quad s = v \tau   \\
			& = & \dfrac{\tau^{\alpha-1}}{\Gamma(1-\alpha)} \int_a^{u \tau} (u \tau - s)^{-\alpha} ( \tau x'(s) ) \, ds \\
			& = & \dfrac{\tau^{\alpha}}{\Gamma(1-\alpha)} \int_a^{t} (t - s)^{-\alpha} x'(s) \, ds, \quad t = u \tau \\
			& = & \tau^\alpha \Drl{a}{t}{\alpha} x(t). 
\end{eqnarray*}

The two methods are related through the following result.
\begin{theorem}
If the fractional operator $d^\alpha / dt^\alpha$ respects condition \eqref{D}, then the two methods are equivalent: 
$x \, : \, t \mapsto x(t)$ is solution of \eqref{EL2} if and only if $\tilde{x} \, : \, u \mapsto x(u \tau)$ is solution of \eqref{EL1}. 
\end{theorem}

\begin{proof}

Partial derivatives of $L$ and $L_n$ verify
\begin{eqnarray}
\partial_1 L_n(x,y,u) & = & \partial_1 L \left( x,\dfrac{y}{\tau}, u \tau \right), \label{D1} \\
\partial_2 L_n(x,y,u) & = & \dfrac{1}{\tau} \, \partial_2 L \left( x, \dfrac{y}{\tau}, u \tau \right). \label{D2}
\end{eqnarray}

From \eqref{D}, \eqref{D1} and \eqref{D2}, we infer
\begin{eqnarray*}
\partial_1 L_n \left( \tilde{x}(u), \dfrac{d^\alpha}{du^\alpha} \tilde{x}(u),u \right) & = & \partial_1 L \left( x(t), \tau^{\alpha-1} \dfrac{d^\alpha}{dt^\alpha} x(t), t \right), \\
\partial_2 L_n \left( \tilde{x}(u), \dfrac{d^\alpha}{du^\alpha} \tilde{x}(u),u \right) & = & \dfrac{1}{\tau} \, \partial_2 L \left( x(t), \tau^{\alpha-1} \dfrac{d^\alpha}{dt^\alpha} x(t), t \right),
\end{eqnarray*} 

which proves that \eqref{EL1} and \eqref{EL2} are similar.

\end{proof}

Therefore, if condition \eqref{D} is verified, nondimensionalized solution $u \in [a / \tau, b / \tau] \mapsto \tilde{x}(u)$ of \eqref{EL1} can be redimensionalized into $t \in [a,b] \mapsto \tilde{x}(t / \tau)$, solution of the dimensionalized and equivalent equation \eqref{EL2}.


\subsection{Fractional constants}

In various litterature dealing with fractional equations using the operator $d^\alpha / dt^\alpha$ (homogeneous to $T^{-\alpha}$), homogeneity is sometimes preserved by introducing \emph{fractional constants} (see \cite{Metzler_Klafter} for example). For a wide class of Lagrangians, the $d^\alpha / dt^\alpha$-embedding of a modified Lagrangian leads to a homogeneous fractional equation involving fractional constants.

We consider the same Lagrangian system as above. The temporal dimension of $L$ is here denoted $T^{n_0}$ ($n_0=-2$ in most of the cases).
We suppose that $L$ is equal to its Laurent series in the variable $v$,
\begin{equation} \label{hypo}
L(x,v,t) = \sum_{i \in \mathbb{Z}} a_i \, f_i(x,t) \, v^i, 
\end{equation} 
where $[a_i] = T^{i+n_0}$. The functions $f_i$ have hence no temporal dimension.

Using the fractional operator $\tau^{\alpha-1} d^\alpha / dt^\alpha$, we obtain for any function $x \in C^\infty([a,b])$
\begin{align}
L \left( x(t), \, \tau^{\alpha-1} \dfrac{d^\alpha}{dt^\alpha} x(t), \,t \right) & = \sum_{i \in \mathbb{Z}} a_i \, f_i(x,t) \, \left( \tau^{\alpha-1} \dfrac{d^\alpha}{dt^\alpha} x(t) \right)^i \nonumber \\
		& = \sum_{i \in \mathbb{Z}} a_i \, f_i(x,t) \tau^{i(\alpha-1)} \left( \dfrac{d^\alpha}{dt^\alpha} x(t) \right)^i \nonumber \\
		& = \sum_{i \in \mathbb{Z}} \bar{a}_i \, f_i(x,t) \left( \dfrac{d^\alpha}{dt^\alpha} x(t) \right)^i, \quad \text{with } \bar{a}_i = a_i \, \tau^{i(\alpha-1)}. \nonumber
\end{align} 

The dimension of $\bar{a}_i$ is $[\bar{a}_i] = T^{i\alpha + n_0}$. With this fractional dimension, those coefficients can be considered as \emph{fractional constants}.

We can form a new Lagrangian with those constants.
\begin{definition}
For a Lagrangian $L(x,v,t) = \sum_{i \in \mathbb{Z}} a_i \, f_i(x,t) \, v^i$, the fractional Lagrangian, denoted $L_f$, is defined by
\begin{equation}
L_f(x,v,t) = \sum_{i \in \mathbb{Z}} \bar{a}_i \, f_i(x,t) \, v^i. \nonumber
\end{equation} 
\end{definition}

The $d^\alpha / dt^\alpha$-embedding of $L_f$, denoted $\hat{L}_f$, verifies
\begin{equation}
\begin{array}{cccl}
\hat{L}_f \, : & C^\infty([a,b]) & \longrightarrow & \quad C^\infty([a,b]) \\
             &       x		& \longmapsto    &  \displaystyle{\sum_{i \in \mathbb{Z}}} \bar{a}_i f_i(x(\bullet),\bullet) \left(\dfrac{d^\alpha}{dt^\alpha} x(\bullet) \right)^i. \nonumber
\end{array} 
\end{equation}   

Hence we have
\begin{equation} \label{link23}
\hat{L} \left( \tau^{\alpha-1} \dfrac{d^\alpha}{dt^\alpha} \right) = \hat{L}_f \left( \dfrac{d^\alpha}{dt^\alpha} \right), 
\end{equation} 

which leads to
\begin{theorem}
The Euler-Lagrange equation associated with the $d^\alpha / dt^\alpha$-embedding of $L_f$ is given, in the causal case, by
\begin{equation} \label{EL3}
\forall t \in [a,b], \quad \sum_{i \in \mathbb{Z}} \bar{a}_i \partial_1 f_i(x,t) \left(\dfrac{d^\alpha}{dt^\alpha} x(t) \right)^i - i\, \bar{a}_i \, \dfrac{d^\alpha}{dt^\alpha} \left( f_i(x,t) \dfrac{d^\alpha}{dt^\alpha} x(t) \right)^{i-1}  = 0,
\end{equation} 
with $x \in C^\infty([a,b])$.
\end{theorem}

With relation \eqref{link23}, we infer the following result.
\begin{theorem}
If condition \eqref{hypo} is verified, the $\left( \tau^{\alpha-1} d^\alpha / dt^\alpha \right)$-embedding of $L_f$ and the $d^\alpha / dt^\alpha$-embedding of $L_f$ are equivalent,

\begin{eqnarray}
EL \left( \mc{E} \left( L, \tau^{\alpha-1} \dfrac{d^\alpha}{dt^\alpha}\right) \right) & \equiv & EL \left( \mc{E} \left( L_f, \dfrac{d^\alpha}{dt^\alpha}\right) \right), \label{equiv1} \\
\mc{E} \left( EL(L), \tau^{\alpha-1} \dfrac{d^\alpha}{dt^\alpha} \right) & \equiv & \mc{E} \left( EL(L_f), \dfrac{d^\alpha}{dt^\alpha} \right). \label{equiv2}
\end{eqnarray}

In the causal case, we have besides \eqref{equiv1} $\equiv$ \eqref{equiv2}.
\end{theorem}

One verifies that equation \eqref{EL3} is homogeneous: inhomogenity of the operator $d^\alpha / dt^\alpha$ is counterbalanced by the fractional constants $\bar{a}_i$.
However, the homogeneity problem is here only shifted: $L_f \left( x(t), d/dt \, x(t), t \right)$ and $EL(L_f)$ are now inhomogeneous. 
Furthermore, the physical relevance of $L_f$ is discutable, and this Lagrangian may appear as an \emph{ad hoc} construction to obtain homogeneous fractional equations. In this sense, this third method does not fit with an underlying idea of the embedding theory. For any Lagrangian system, its Lagrangian is independant of the dynamics: it is only the choice of the evolution operator, i.e. the temporal derivative, which conditions the dynamics of the system.

Finally, it seems that the second method is the most relevant homogeneous fractional embedding.


\section{Natural Lagrangian systems} \label{natural}


\subsection{General case}

An illustration of those methods is given here with the most widespread type of Lagrangians: the natural ones. 

\begin{definition}
We consider an autonomous Langrangian system with Lagrangian $L(x,v)$, evolving on a temporal interval $[a,b]$. $L$ is a natural Lagrangian if it can be written as
\begin{equation} \label{lag_nat}
L(x,v) = \dfrac{1}{2}m v^2 - V(x),
\end{equation} 
where $m$ is the mass of the system, $\frac{1}{2}m v^2$ is the kinetic energy, and $V(x)$ the potential energy. The temporal dimension of $L$ is $T^{-2}$.
\end{definition}

We introduce a time constant $\tau$, and we apply now the three embeddings to those systems.

\subsubsection{Temporal nondimensionalization}

We denote $\tilde{a} = a / \tau$ and $\tilde{b} = b / \tau$.

The new Lagrangian is
\begin{equation}
L_n(x,y) = L \left( x, \dfrac{y}{\tau} \right) = \dfrac{1}{2 \tau^2} m y^2 - V(x), \nonumber
\end{equation}  

and its embedding verifies
\begin{equation}
\forall x \in C^\infty([\tilde{a}, \tilde{b}]), \, \forall u \in [\tilde{a}, \tilde{b}], \quad \hat{L}_n(x)(u) = \dfrac{1}{2 \tau^2} m \left( \dfrac{d^\alpha}{du^\alpha} x(u) \right)^2 - V \left( x(u) \right). \nonumber
\end{equation} 

The associated action is
\begin{align}
\mc{A}(\hat{L}_n)(x) & = \int_{\tilde{a}}^{\tilde{b}} L_n \left( x(u), \dfrac{d^\alpha}{du^\alpha} x(u) \right) \, du \nonumber \\                              
		& = \int_{\tilde{a}}^{\tilde{b}} \dfrac{1}{2 \tau^2} m \left( \dfrac{d^\alpha}{du^\alpha} x(u) \right) ^2 - V \left( x(u) \right) \, du, \nonumber
\end{align} 

and the Euler-Lagrange equation, in the causal case, turns into
\begin{equation}
\forall u \in [\tilde{a}, \tilde{b}], \quad \dfrac{m}{\tau^2} \dfrac{d^{2 \alpha}}{du^{2 \alpha}} x(u) + V' \left( x(u) \right) = 0. \nonumber
\end{equation}

\subsubsection{Homogeneous fractional derivative}

The initial Lagrangian remains inchanged. 
Its embedding is now
\begin{equation}
\forall x \in C^\infty([a,b), \, \forall t \in [a,b], \quad \hat{L}_h(x)(t) = \dfrac{1}{2} m \tau^{2(\alpha-1)} \left( \dfrac{d^\alpha}{dt^\alpha} x(t) \right)^2 - V \left( x(t) \right), \nonumber
\end{equation}

the action is given by
\begin{align}
\mc{A}(\hat{L}_h)(x) & = \int_a^b L \left( x(t), \tau^{\alpha-1} \dfrac{d^\alpha}{dt^\alpha} x(t), t \right) \, dt \nonumber \\
                & = \int_a^b \dfrac{1}{2}m \tau^{2\alpha-2} \left( \dfrac{d^\alpha}{dt^\alpha} x(t) \right)^2 - V \left( x(t) \right) \, dt, \nonumber
\end{align} 

and the causal Euler-Lagrange equation is
\begin{equation} \label{EL2_nat}
\forall t \in [a,b], \quad m \tau^{2 (\alpha - 1)} \dfrac{d^{2 \alpha}}{dt^{2 \alpha}} x(t) + V' \left( x(t) \right) = 0.
\end{equation}

\subsubsection{With fractional constants}

Let $x_0$ be a position such as $V(x_0) \neq 0$. The potential energy can hence be written as $V(x) = a_0 \, f_0(x,t)$, with $a_0 = V(x_0)$, which temporal dimension is $T^{-2}$, and $f_0(x,t) = V(x) / V(x_0)$, which is dimensionless.

The Lagrangian \eqref{lag_nat} can therefore be expressed in the form \eqref{hypo}, with
\begin{itemize}
\item $a_i = 0$ for $i < 0$,
\item $a_0 = V(x_0)$, $f_0(x,t) = \dfrac{V(x)}{V(x_0)}$,
\item $a_1 = 0$,
\item $a_2 = m$, $f_2(x,t) = \dfrac{1}{2}$,
\item $a_j = 0$ for $j > 2$.
\end{itemize}

The fractional constants are 
\begin{itemize}
\item $\bar{a}_0 = a_0 = V(x_0)$,
\item $\bar{a}_2 = a_2 \, \tau^{2(\alpha-1)} = m \, \tau^{2(\alpha-1)}$,
\end{itemize}

and lead to the new fractional Lagrangian
\begin{equation}
L_f(x,v) = \dfrac{1}{2} m \, \tau^{2(\alpha-1)} v^2 - V(x). \nonumber
\end{equation} 

The associated action is defined by
\begin{equation}
\mc{A}(\hat{L}_f)(x) = \int_a^b  \dfrac{1}{2} m \tau^{2(\alpha-1)} \left( \dfrac{d^\alpha}{dt^\alpha} x(t) \right)^2 - V(x(t)) \, dt. \nonumber
\end{equation} 

Its extremum provides the following Euler-Lagrange equation:
\begin{equation}
\forall t \in [a,b], \quad m \tau^{2 (\alpha - 1)} \dfrac{d^{2 \alpha}}{dt^{2 \alpha}} x(t) + V' \left( x(t) \right) = 0, \nonumber
\end{equation} 

which is similar to \eqref{EL2_nat}.

Now, let us solve those equations in the case of the harmonic oscillator.


\subsection{Harmonic oscillator} \label{sec_osc}

This system has already been studied in the fractional case (see \cite{Stan, Zasl_Stan, Zasl_Tarasov}).

The harmonic oscillator is a natural Lagrangian system which Lagrangian $L$ is given by:
\begin{equation} \label{lag_osc}
L(x,v) = \dfrac{1}{2}m v^2 - \dfrac{1}{2}k x^2,
\end{equation} 
with $[x] = L$, $[v] = L \, T^{-1}$, $[m] = M$ et $[k] = M \, T^{-2}$.

As above, we introduce a time constant $\tau$ (which is \emph{a priori} not linked with $\sqrt{m / k}$).
We arbitrarily choose here the Caputo derivative for the fractional operator $d^\alpha / dt^\alpha$.

\subsubsection{Temporal nondimensionalization}

The nondimensionalized Lagrangian $L_n$ is:
\begin{equation}
L_n(x,y) = \dfrac{1}{2 \tau^2}m y^2 - \dfrac{1}{2}k x^2. \nonumber
\end{equation} 

Its embedding, $\hat{L}_n(x)(u) = \frac{1}{2 \tau^2} m \left( d^\alpha / du^\alpha \, x(u) \right)^2 - \frac{1}{2} k x(u)^2$ leads to the following Euler-Lagrange equation:
\begin{equation}
\forall u \in [\tilde{a}, \tilde{b}], \quad \dfrac{m}{\tau^2} \dfrac{d^{2 \alpha}}{du^{2 \alpha}} x(u) + k \, x(u) = 0. \nonumber
\end{equation} 

With the usual notation $\omega = \sqrt{k / m}$, we obtain
\begin{equation} 
\forall u \in [\tilde{a}, \tilde{b}], \quad \dfrac{d^{2 \alpha}}{du^{2 \alpha}} x(u) + (\omega \tau)^2 x(u) = 0. \nonumber
\end{equation} 

The (nondimensionalized) solution, denoted $x_n$, is given by
\begin{equation}
\forall u \in [\tilde{a}, \tilde{b}], \quad x_n(u) = x_n(\tilde{a}) E_{2 \alpha} \left( - (\omega \tau)^2 (u-\tilde{a})^{2 \alpha} \right), \nonumber
\end{equation} 

where, for $\lambda > 0$, $E_{\lambda}$ is the Mittag-Leffler function defined on $\mathbb{C}$ by
\begin{equation}
\forall z \in \mathbb{C}, \quad E_{\lambda}(z) = \sum_{k=0}^\infty \frac{z^k}{\Gamma(\lambda k + 1)}. \nonumber
\end{equation} 

The dimensionalized solution, denoted $x_d$, is hence
\begin{equation}
\forall t \in [a,b], \quad x_d(t) = x_n(\dfrac{t}{\tau}) = x_d(a) E_{2 \alpha} \left( - \omega^2 \tau^{2(1-\alpha)} (t- a)^{2 \alpha} \right). \nonumber
\end{equation} 

One verifies that the argument of the Mittag-Leffler function is dimensionless.

\subsubsection{Homogeneous fractional derivative}

The embedded Lagrangian $\hat{L}_h$, given by
\begin{equation}
\hat{L}_h = \dfrac{1}{2} m \tau^{2(\alpha-1)} \left( \dfrac{d^\alpha}{dt^\alpha} x(t) \right)^2 - k \, x(t)^2, \nonumber
\end{equation} 
provides the equation
\begin{equation} \label{EL_osc_2}
\forall t \in [a,b], \quad m \tau^{2 (\alpha - 1)} \dfrac{d^{2 \alpha}}{dt^{2 \alpha}} x(t) + k \, x(t) = 0.
\end{equation} 

Introducing $\omega$, we obtain
\begin{equation} 
\forall t \in [a,b], \quad \dfrac{d^{2 \alpha}}{dt^{2 \alpha}} x(t) + \omega^2 \tau^{2(1-\alpha)} x(t)  = 0, \nonumber
\end{equation} 

which admits the solution $x_h$ defined by
\begin{equation}
x_h(t) = x_h(a) E_{2 \alpha} \left( - \omega^2 \tau^{2(1-\alpha)} (t- a)^{2 \alpha} \right). \nonumber
\end{equation} 

One verifies that $x_n = x_h$.

\subsubsection{With fractional constants}

The Lagrangian \eqref{lag_osc} can be written as \eqref{hypo}, with
\begin{itemize}
\item $a_0 = k$, $f_0(x,t) = \frac{1}{2} x^2$,
\item $a_1 = 0$,
\item $a_2 = m$, $f_2(x,t) = \frac{1}{2}$,
\item $a_i = 0$ for $i \geq 3$.
\end{itemize}

The fractional constants are given by
\begin{itemize}
\item $\bar{a}_0 = k$,
\item $\bar{a}_2 = a_2 \, \tau^{2(\alpha-1)} = m \, \tau^{2(\alpha-1)}$.
\end{itemize}

The associated fractional Lagrangian, $L_f$, is
\begin{equation}
L_f(x,v) = \frac{1}{2} m \tau^{2(\alpha-1)} v^2 - \frac{1}{2} k \, x^2. \nonumber
\end{equation}  

We obtain the Euler-Lagrange equation,
\begin{equation}
\forall t \in [a,b], \quad m \tau^{2(\alpha-1)} \dfrac{d^{2 \alpha}}{dt^{2 \alpha}} x(t) + k \, x(t) = 0, \nonumber
\end{equation} 

which is similar to equation \eqref{EL_osc_2}.

\subsubsection{Comparison with the inhomogeneous fractional embedding}

If we directly apply the fractional embedding presented in \cite{Cresson} to the Lagrangian \eqref{lag_osc}, we obtain
\begin{equation} \label{osc_inhom}
\dfrac{d^{2 \alpha}}{dt^{2 \alpha}} x(t) + \omega^2 x(t) = 0.
\end{equation} 

The temporal dimension of $d^{2 \alpha} / dt^{2 \alpha} \, x(t)$ is $T^{-2\alpha}$, whereas for $\omega^2 x(t)$, it is $T^{-2}$.

The solution $x_i$ of \eqref{osc_inhom} is
\begin{equation}
x_i(t) = x_i(a) E_{2 \alpha} \left( - \omega^2 (t- a)^{2 \alpha} \right). \nonumber
\end{equation} 

The argument of the Mittag-Leffler function is $T^{2(\alpha-1)}$, whereas it should be dimensionless.


\section{Conclusion}

With the first two embeddings presented above (temporal nondimensionalization and homogeneous fractional derivative), it becomes possible to preserve homogeneity of embedded equations. Both of them lead to the same fractional equation. Therefore, from any Lagrangian system, it is possible to construct a fractional equation which respects two fundamental physical principles: homogeneity and least action principle.

To do so, a time constant $\tau$ has to be introduced. Similarly to the fractional exponent $\alpha$, this parameter is not defined by the equation itself. Contrary to the nondimensionalization of classical equations, fractional solutions depend on $\tau$ (see the example of the harmonic oscillator). This time constant is hence more than a calculus intermediary: it conditions the dynamcis of the system. Its physical relevance seems as much important as the constants of the initial equation. To sum up, a fractional equation should be characterized by two parameters: the fractional exponent $\alpha$ and the time constant $\tau$.

So as to understand the physical meaning of this characteristic time, it may be useful to set the context in which such equations arise. In \cite{Hilfer_FFD, Zasl_HCFD}, the fractional aspect of the dynamics appears only on long time scales. At a microscopic level, the dynamics is classical and evolves according to a microscopic time, but when one zooms out at a macroscopic level, ruled by a new macroscopic time, the fractional aspect may emerge. It might be through the relation between these two time scales that the time constant $\tau$ may appear.
 

\section*{Acknowledgements}

The author thanks his Ph.D. director, J. Cresson, for his support and for fruitful discussions, W. Thuillot, director of the IMCCE and J.E. Arlot, in charge of the GAP section, for their hospitality.


\bibliographystyle{plain}
\bibliography{Biblio_Homogeneity}

\end{document}